\documentclass[a4paper]
{article}

\usepackage{amssymb,amsmath, txfonts,mathrsfs}
\usepackage{xcolor}
\makeindex

\input xy  \xyoption{all}

\newtheorem{deff}{Definition}[section]
\newtheorem{lemma}[deff]{Lemma}

\newtheorem{theorem}[deff]{Theorem}

\newtheorem{corollary}[deff]{Corollary}

\newtheorem{proposition}[deff]{Proposition}

\newtheorem{em-example}[deff]{Example}
\newtheorem{em-examples}[deff]{Examples}
\newtheorem{em-def}[deff]{Definition}        
\newtheorem{em-remark}[deff]{Remark}         
\newtheorem{em-remarks}[deff]{Remarks}
\newtheorem{em-question}[deff]{Question}
\newtheorem{em-openquestion}[deff]{Open question}

\newtheorem{claim}[deff]{Claim}

\newenvironment{definition}{\begin{em-def} \em  }{ \end{em-def}}
\newenvironment{openquestion}{\begin{em-openquestion} \em }{\end{em-openquestion}}
\newenvironment{remark}{\begin{em-remark} \em }{\end{em-remark}}

\newenvironment{proof}{\noindent {\it Proof}:}{\QED \smallskip}

\newcommand{\dis}{\displaystyle}
\newcommand{\D}{{\rm d}}

\newcommand{\tr}{^{\triangleright}}

\newcommand\QED{\hfill QED \medskip}

\newcommand{\abs}[1]{\left| #1 \right|}

\newcommand{\Graev}[1]{\mathcal{T}_{\{#1\}}}

\def\CHom{\mathop{\rm CHom}}

\def\:{\nobreak \hskip .1111em\mathpunct {}\nonscript \mkern
-\thinmuskip {:}\hskip .3333emplus.0555em\relax}

\catcode`\@=12
\def\T{{\mathbb T}}
\def\Q{{\mathbb Q}}

\def\Z{{\mathbb Z}}
\def\N{{\mathbb N}}
\def\R{{\mathbb R}}
\def\Q{{\mathbb Q}}

\def\G{\mathcal{G}}

\def\taub{\tau_\mathbf{b}}

\def\lambdab{\lambda_\mathbf{b}}

\def\deltab{\delta_\b}

\def\b{\mathbf{b}}

\def\Zbinfty{\Z(\b^\infty)}

\def\Gammap{\Gamma_\mathbf{p}}
\def\Gammab{\Gamma_\b}
\def\tcu{topology of uniform convergence }

\def\ro{\rho}
\def\Dlinfty{{\mathcal D}_\infty^\ell}
\def\D{\mathcal{D}}
\def\Dinfty{\mathcal{D}_\infty}
\def\LQC{\mathcal{L}qc}
\def\MAP{\mathcal{MAP}}

\def\B
{\overline{\mathcal{V}}_{abs}}
\def\tauc{\tau_\mathbf{c}}

\title{Non metrizable topologies on $\Z$ with countable dual group.}
\author{ D. de la Barrera Mayoral\footnote{ORCID: 0000-0002-0024-5265. e-mail: dbarrera@mat.ucm.es.  Partially supported by Ministerio de Econom\'ia y Competitividad grant: MTM2013-42486-P.}}

\begin{document}

\maketitle

\begin{abstract}
In this paper we give two families of non-metrizable topologies on the group of the integers having a countable dual group which is isomorphic to a infinite torsion subgroup of the unit circle in the complex plane. Both families are related to $D$-sequences, which are sequences of natural numbers such that each term divides the following. The first family consists of locally quasi-convex group topologies. The second consists of complete topologies which are not locally quasi-convex. In order to study the dual groups for both families we need to make numerical considerations of independent interest.

\end{abstract}



\section{Introduction and Notation.}

The most important class of topological abelian groups are the locally compact groups. In this class, the most interesting theorems can be proved. In this paper we consider a wider class of groups, which is the class of locally quasi-convex groups. Clearly, locally compact groups are locally quasi-convex, since they can be considered as a dual group (in fact, a locally compact abelian group is isomorphic to its bidual group) and dual groups are locally quasi-convex.

This paper is part of a project about duality, which intends to solve, among others, the Mackey Problem (first stated in \cite{ChaMarTar1999}). This problem asks whether there exists a maximum element in the family of all locally quasi-convex compatible topologies (see Definition \ref{Def-G-Mackey}). 

We present two different kinds of topologies on the group of the integers, having as common feature that topologies in both families have countable dual group, which is isomorphic to an infinite torsion subgroup of the unit circle in the complex plane:

 On the one hand, the topologies presented in Section \ref{SeccionGammaB} are metrizable, non-complete and locally quasi-convex and the results contained are the natural further steps of the ones given in \cite{AusBar2012}.

 On the other hand, the topologies presented in Section \ref{SeccionGraev} are complete, non-metrizable but they are not locally quasi-convex. These topologies, introduced by Graev, were thoroughly studied by Protasov and Zelenyuk. Precisely, in \cite{ProZel1999}, a neighborhood basis for this topology is given and this allows us to study these topologies on $\Z$ from a numerical point of view. The fact that these topologies have countable dual group is surprising due to the belief that the more open set the topology has, the more continuous characters it should have.

All the groups considered will be abelian and the notation will be additive.

In order to state properly the Mackey Problem, we need first to recall some notation on duality:

\begin{definition}

We will consider the unit circle as $$\T=\R/\Z=\left\{x+\Z:x\in\left(-\frac{1}{2},\frac{1}{2}\right]\right\}.$$
 We will consider in $\T$ the following neighborhood basis: $$\T_m:=\left\{x+\Z:x\in\left[-\frac{1}{4m},\frac{1}{4m}\right]\right\}.$$
 We set $$\T_+:=\left\{x+\Z:x\in\left[-\frac{1}{4},\frac{1}{4}\right]\right\}=\T_1.$$

\end{definition}

\begin{definition}

Let $(G,\tau)$ be a topological group. Denote by $G^\wedge$ the set of continuous homomorphisms (or continuous characters) $\CHom(G,\T)$. This set is a group when we consider the operation $(f + g)(x)= f(x) + g(x)$. In addition, if we consider the compact open topology on $G^\wedge$ it is a topological group. The topological group $G^\wedge$ is called the \textbf{dual group of $G$.}

\end{definition}

\begin{definition}

The natural embedding $$\alpha_G:G\rightarrow G^{\wedge\wedge},$$\index{$\alpha_G$} is defined by $$x\mapsto\alpha_G(x):G^\wedge\rightarrow\T,$$ where $$\chi\mapsto\chi(x).$$

\end{definition}

\begin{proposition}

$\alpha_G$ is a homomorphism.

\end{proposition}

\begin{proof}

$\alpha_G(x+y)(\chi)=\chi(x+y)=\chi(x)\chi(y)=\alpha_G(x)(\chi)\alpha_G(y)(\chi)= (\alpha_G(x)\alpha_G(y))(\chi)$, for all $\chi\in G^\wedge$.

Thus, $\alpha_G(x+y)=\alpha_G(x)\alpha_G(y)$

\end{proof}

\begin{definition}

A topological group $(G,\tau)$ is {\bf reflexive} if $\alpha_G$ is a topological isomorphism.

\end{definition}

\begin{definition}

Let $\tau,\nu$ be two group topologies on a group $G$. We say that $\tau$ and $\nu$ are {\bf compatible} if algebraically $(G,\tau)^\wedge=(G,\nu)^\wedge$.

\end{definition}

Now we include the definition of quasi-convex set, which emulates the Hahn-Banach Theorem from Functional Analysis:

\begin{definition}\cite{Vil51}

Let $(G,\tau)$ be a topological group and let $M\subset G$ be a subset. We say that $M$ is {\bf quasi-convex} if for every $x\in G\setminus M$, there exists $\chi \in G^\wedge$ such that $\chi(M)\subset \T_+$ and $\chi(x)\notin\T_+$. We say that $\tau$ is {\bf locally quasi-convex} if there exists a neighborhood basis for $\tau$ consisting of quasi-convex subsets. We denote by {\bf $\LQC$} the class of all locally quasi-convex topological groups.

\end{definition}

\begin{definition}

Let $\tau$ be a group topology. We define the locally quasi-convex modification of $\tau$ as the finest locally quasi-convex group topology $(\tau)_{lqc}$ among all those satisfying that they are coarser than $\tau$.

\end{definition}

For a deeper insight in the locally quasi-convex modification of a topological group, see \cite{DiaMar2014}.

\begin{lemma}

Let $(G,\tau)$ be a topological group. Then $(G,\tau)^\wedge=(G,(\tau)_{lqc})^\wedge$.

\end{lemma}

\begin{definition}

Let $(G,\tau)$ a topological group. The \textbf{polar} of a subset $S\subset
G$ is defined as $$S\tr: =\{\chi\in G^\wedge\mid\chi(S)\subset\T_+\}.$$

\end{definition}

\begin{remark}

It is clear that $$(G,\tau)^\wedge=\bigcup_{V\in\mathcal{B}}V\tr,$$ where $\mathcal{B}$ is a neighborhood basis of $0$ for the topology $\tau$.

\end{remark}

\begin{definition}\cite{DikMarTar2014}\label{Def-G-Mackey}

Let $\G$ be a class of abelian topological groups and let $(G, \tau)$ be a topological group in $\G$. Let $\mathcal{C}_\G(G,\tau)$ denote the family of all $\G$-topologies $\nu$ on $G$ compatible with $\tau$.
We say that $\mu\in \mathcal{C}_\G(G,\tau)$ is the {\bf $\G$-Mackey topology} for $(G,G^\wedge)$ (or the {\bf Mackey topology for $(G,\tau)$ in $\G$}) if $\nu\leq\mu$ for all $\nu\in \mathcal{C}_\G(G,\tau)$.

If $\G$ is the class of locally quasi-convex groups $\LQC$, we will simply say that the topology is Mackey.

\end{definition}

The problem whether the Mackey Topology for a topological group $(G,\tau)$ exists was first posed in \cite{ChaMarTar1999}. This problem is called the Mackey Problem. Along the last years several mathematicians have tried to solve the Mackey Problem: de Leo \cite{TesisLorenzo}, Aussenhofer, Dikranjan, Mart\'in-Peinador \cite{ADM1}, D\'iaz-Nieto \cite{DiaMar2014}, Gabriyelyan \cite{DikGabTar2014}. A survey including similarities and differences between the Mackey topology in spaces and the Mackey Problem for groups has recently appeared \cite{MarTar2015}. In order to solve the Mackey Problem (in the negative) we try to find all the locally quasi-convex group topologies which are compatible with a non-discrete linear topology on $\Z$ and study if the supremum of all these topologies is again compatible. So far, it is not known if the supremum of compatible topologies is again compatible.

Now we give the basic definitions on $D$-sequences, which are sequences of natural numbers that encode the relevant information of non-discrete linear topologies on $\Z$.

\begin{definition}\index{$D$-sequence}

A sequence of natural numbers $\b=(b_n)_{n\in\N_0}\in\N^{\N_0}$ is called a \textbf{$D$-sequence} if it satisfies:

\begin{itemize}

\item[(1)] $b_0 = 1$,

\item[(2)] $b_n\neq b_{n+1}$ for all $n\in\N_0$,

\item[(3)] $b_n$ divides $b_{n+1}$ for all $n\in\N_0$.

\end{itemize}

The following notation will be used in the sequel:

\begin{itemize}

\item $\D:=\{\b : \b \mbox{ is a } D-\mbox{sequence}\}$.

\item $\Dinfty:=\{\b\in\mathcal{D}:\frac{b_{n+1}}{b_n}\rightarrow\infty\}$.

\item $\Dlinfty:=\{\b\in\mathcal{D}: \frac{b_{n+\ell}}{b_n}\rightarrow\infty\}$.
\item $\Dinfty(\b):=\{\mathbf{c}: \mathbf{c} \mbox{ is a subsequence of }\b \mbox{ and }\mathbf{c}\in\Dinfty\}$.

\item $\Dlinfty(\b):=\{\mathbf{c}: \mathbf{c} \mbox{ is a subsequence of }\b \mbox{ and }\mathbf{c}\in\Dlinfty\}$.

\end{itemize}
\end{definition}

Our interest on $D$-sequences stems from the fact that several compatible topologies on $\Z$ can be associated to $D$-sequences. This family of topologies is interesting due to the fact that we don't know in general, if the supremum of compatible topologies is again compatible. This is the main difference between the study of the Mackey Topology in spaces and the Mackey Problem for groups.

\begin{definition}\label{DefinitionQn}

Let $\b$ be a $D$-sequence. Define $(q_n)_{n\in\N}$ by $q_n:=\frac{b_n}{b_{n-1}}$.

\end{definition}

\begin{definition}

Let $\b$ be a $D$-sequence. We say that
\begin{itemize}

\item[(a)] $\b$ has \textbf{bounded ratios} if there exists a natural number $N$, satisfying that $q_n=\frac{b_n}{b_{n-1}}\leq N$.

\item[(b)] $\b$ is \textbf{basic} if $q_n$ is a prime number for all $n\in\N$.

\end{itemize}

\begin{definition}

Let $\b$ be a $D$-sequence. We write $$\Z(b_n):=\left\{\frac{k}{b_n}+\Z:k=0,1,\dots,b_n-1\right\}\mbox{ and }\Zbinfty:=\bigcup_{n\in\N_0}\Z(b_n).$$
As quotient groups, we can write $$\Z(b_n)=\left\langle\left\{ \frac{1}{b_n}\right\}\right\rangle/\Z\mbox{ and }\Zbinfty=\left\langle \left\{\frac{1}{b_m}:m\in\N_0\right\}\right\rangle/\Z.$$\index{$\Zbinfty$}
In other words, $\Z(b_n)$ consists of the elements whose order divides $b_n$ and $\Zbinfty$ is generated by the elements having order $b_n$ for some natural number $n\in\N_0$.

\end{definition}

\end{definition}

\begin{proposition}\label{bn}

Let \textbf{b} be a $D$-sequence. Suppose that $q_{j+1}\neq 2$ for infinitely many $j$. For each integer number $L\in\Z$, there exists a natural number $N=N(L)$ and \textbf{unique} integers $k_0,\dots ,k_{N}$, such that:

\begin{enumerate}

\item[(1)] $\dis L=\sum_{j=0}^{N}k_jb_j$.

\item[(2)] $\dis \left|\sum_{j=0}^{n}k_jb_j\right|\leq\frac{b_{n+1}}{2}$ for all $n$.

\item[(3)] $\dis k_j\in\left(-\frac{q_{j+1}}{2},\frac{q_{j+1}}{2} \right]$, for $0\leq j\leq N$.

\end{enumerate}

\end{proposition}

The proof of Proposition \ref{bn} can be found in \cite[Proposition 2.2.1]{TesisDaniel} and \cite[Proposition 1.4]{Bar2014}

\begin{definition}\label{DefinitionBadicas}
The family
$$\mathcal{B}_\b:=\{b_n\Z:n\in\N_0\}$$
is a neighborhood basis of $0$ for a linear (that is, it has a neighborhood basis consisting in open subgroups) group topology on $\Z$, which will be called $\b$-adic topology and is denoted by $\lambdab$.
\end{definition}

This topology is precompact. In \cite{AusBar2012}, it is proved that $(\Z,\lambdab)^\wedge=\Zbinfty$ and the following

\begin{lemma}\label{bConvergeEnLambdab}

Let $\b$ be a $D$-sequence. Then $b_n\rightarrow0$ in $\lambdab$.

\end{lemma}

\begin{definition}

A $D$-sequence $\b$ in $\Z$ induces in a natural way a sequence in $\T$. Namely, for a $D$-sequence, $\b=(b_n)_{n\in\N_0}$, define
$$\underline{\b}:=\left(\frac{1}{b_n}+\Z:n\in\N_0\right)\subset\T.$$
 Denote by $\taub$ the topology of uniform convergence on $\underline{\b}$.

\end{definition}

\section{Killing $\lambdab$ null sequences.}\label{SeccionKilling}

In this section we prove that for a $D$ sequence, $\b$ with bounded ratios and any $\lambdab$ convergent sequence, $(a_n)$, there exists a strictly finer compatible topology (constructed as a topology of uniform convergence on a suitable sequence in $\Dinfty(\b)$) such that the fixed sequence $(a_n)$ is no longer convergent. This will allow us to find a locally quasi-convex topology which has no convergent sequences, but it is still compatible.

\begin{theorem}\label{proposicion20121128}\label{proposicion20130614}

Let $\b$ be a basic $D$-sequence with bounded ratios and let $(x_n)\subset\Z$ be a non-quasiconstant sequence such that $x_n\stackrel{\lambdab}{\rightarrow}0$. Then there exists a metrizable locally quasi-convex compatible group topology $\tau(=\tauc$ for some subsequence $\mathbf{c}\mbox{ of }\b)$ on $\Z$ satisfying:
\par
$(a)$ $\tau$ is compatible with $\lambdab$.
\par
$(b)$  $x_n\stackrel{\tau}{\nrightarrow}0$.
\par
$(c)$ $\lambdab<\tau$.

\end{theorem}

\begin{proof}

We define $\tau$ as the \tcu on a subsequence $\mathbf{c}$ of $\b$, which we construct by means of the claim:

\begin{claim}

Since $\b$ is a $D$-sequence with bounded ratios, there exists $L\in\N$ such that $q_{n+1}\leq L$ for all $n\in\N$.

We can construct inductively two sequences $(n_j),(m_j)\subset \N$ such that $$n_{j+1}-n_j>j\mbox{ and }\dis \frac{x_{m_j}}{b_{n_j+1}}+\Z\notin\T_L.$$

\end{claim}

\underline{Proof of the claim:}

\underline{$j=1.$}

Choose $n_1\in\N$ such that there exists $\dis x_{m_1}$, satisfying $\dis b_{n_1}\mid x_{m_1}$ but $\dis b_{n_1+1}\nmid x_{m_1}$. By Proposition \ref{bn}, we write $$\dis x_{m_1}=\sum_{i=0}^{N(x_{m_1})}k_ib_{i}.$$
The condition $\dis b_{n_1}\mid x_{m_1}$ implies $k_i=0$ if $i<n_1$ and $\dis b_{n_1+1}\nmid x_{m_1}$ implies $k_{n_1}\neq 0$.

Hence, $$\dis \frac{x_{m_1}}{b_{n_1+1}}+\Z= \frac{\dis \sum_{i=0}^{N(x_{m_1})}k_ib_{i}}{b_{n_1+1}}+\Z = \frac{\dis \sum_{i=0}^{n_1}k_ib_{i}}{b_{n_1+1}}+\Z = \frac{k_{n_1}}{q_{n_1+1}}+\Z.$$
 We have that $$\frac{1}{2}\ge\frac{\abs{ k_{n_1}} }{q_{n_1+1}}\geq \frac{1}{L}.$$
 Therefore, $$\dis\frac{x_{m_1}}{b_{n_1+1}}+\Z\notin\T_L=\left[-\frac{1}{4L},\frac{1}{4L}\right]+\Z.$$

\underline{$j\Rightarrow j + 1.$}

Suppose we have $n_j, m_j$ satisfying the desired conditions. Let $n_{j+1}\geq n_j+j$ be a natural number such that there exists $\dis x_{m_{j+1}}$ satisfying $\dis b_{n_{j+1}}\mid x_{m_{j+1}}$ and $b_{n_{j+1}+1}\nmid x_{m_{j+1}}$. By Proposition \ref{bn}, we write $$\dis x_{m_{j+1}}=\sum_{i=0}^{N(x_{m_{j+1}})}k_ib_{i}.$$
 The condition $\dis b_{n_{j+1}}\mid x_{m_{j+1}}$ implies that $k_i=0$ if $i<n_{j+1}$ and $b_{n_{j+1}+1}\nmid x_{m_{j+1}}$ implies $k_{n_{j+1}}\neq 0$.

Then, $$\dis\frac{x_{m_{j+1}}}{\dis b_{n_{j+1}+1}}+\Z= \frac{\dis \sum_{i=0}^{N(x_{m_{j+1}})}k_ib_{i}}{\dis b_{n_{j+1}+1}}+\Z=\frac{\dis \sum_{i=0}^{n_{j+1}}k_ib_{i}}{\dis b_{n_{j+1}+1}}+\Z = \frac{\dis k_{n_{j+1}}}{\dis q_{n_{j+1}+1}}+\Z.$$
From $$\dis \frac{1}{2}\ge \frac{\abs{k_{n_{j+1}}} }{q_{n_{j+1}+1}}\geq\frac{1}{q_{n_{j+1}+1}}\geq\frac{1}{L},$$
 we get that $$\dis \frac{x_{m_{j+1}}}{b_{n_{j+1}+1}}+\Z\notin\T_L.$$

This ends the proof of the claim.

\vspace{0.3cm}

We continue the proof of Theorem \ref{proposicion20130614}. Consider now $\mathbf{c}=\left(b_{n_j+1}\right)_{j\in\N}$. Let $$\underline{\mathbf{c}}=\left\{\frac{1}{b_{n_j+1}}+\Z : j\in\N\right\}$$
 and let $\tau=\tauc$ be the \tcu on $\underline{\mathbf{c}}$. Then:

\begin{enumerate}

\item[(1)] By Proposition \cite[Remark 3.3]{AusBar2012}, $\tau$ is metrizable and locally quasi-convex. By \cite[Proposition 3.7]{AusBar2012}, we have $\lambdab<\tau$.

\item[(2)] By the claim, we have proved that $x_{m_j}\stackrel{j\rightarrow\infty}{\nrightarrow}0$ in $\tau$. This implies that $x_n\nrightarrow 0$ in $\tau$.

\item[(3)] Since $n_{j+1}\geq n_j+j$, we get that $n_{j+1}-n_j\geq j$. Thus, $\dis \frac{b_{n_{j}+1}}{b_{n_{j+1}+1}}\leq\frac{1}{2^j}\rightarrow 0$. By \cite[Theorem 4.4]{AusBar2012}, we have that $(\Z,\tau)^\wedge=\Zbinfty$.

\end{enumerate}

\end{proof}

\section{The topology $\Gammab$.}\label{SeccionGammaB}

In Section \ref{SeccionKilling}, we use a subfamily of $\Dinfty(\b)$ to eliminate all $\lambdab$ convergent sequences. Unfortunately, this subfamily depends on the choices made in Theorem \ref{proposicion20121128}. This is why we introduce a new topology $\Gammab$, which is the supremmum of the topologies of uniform convergence on sequences in $\Dinfty(\b)$. The topology $\Gammab$ is well-defined (in the sense that it does not depend on the choices made in Theorem \ref{proposicion20121128}).

\begin{definition}

Let $\b$ be a basic $D$-sequence with bounded ratios. We define on $\Z$ the topology $$\Gammab:=\sup\{\tauc:\mathbf{c}\in\Dinfty(\b)\}.$$

\end{definition}

Now we set some results on $\Gammab$:

\begin{remark}

Since $\Dinfty(\b)$ is no-empty, we have that $\lambdab<\Gammab$.

\end{remark}

\begin{theorem}\label{mupNoTieneSucesionesConvergentes}

Let $\b$ a basic $D$-sequence with bounded ratios. Then the topology $\Gammab$ has \textbf{no} nontrivial convergent sequences. Hence, it is not metrizable.

\end{theorem}

\begin{proof}

 Since every $\Gammab$-convergent sequence is $\lambdab$-convergent, we consider only $\lambdab$-convergent sequences. Let $(x_n)\subseteq\Z$ be a sequence such that $x_n\stackrel{\lambdab}{\rightarrow}0$. By Proposition \ref{proposicion20121128}, there exists a subsequence $(b_{n_k})_k$ satisfying that $\tau_{(b_{n_k})}$ is a locally quasi-convex topology, $(b_{n_k})\in\Dinfty(\b)$ and $\dis x_n\stackrel{\tau_{b_{n_k}}}{\nrightarrow}0$. Since $\tau_{b_{n_k}}\leq\Gammab$, we know that $x_n\stackrel{\Gammab}{\nrightarrow}0$. Hence, the only convergent sequences in $\Gammab$ are the trivial ones. Since $\Gammab$ has no nontrivial convergent sequences, it is {\bf not metrizable.}

\end{proof}

\begin{corollary}

Let $\mathbf{p}=(p^n)$. Then $\Gammap$ has no nontrivial convergent sequences.

\end{corollary}

\begin{lemma}\label{LemaDualContable}

Let $\b$ be a $D$-sequence. Then $(\Z,\Gammab)^\wedge=\Zbinfty$.

\end{lemma}

For the proof of Lemma \ref{LemaDualContable} we need the following definition of \cite{TesisDaniel}:

\begin{definition}\cite[Definition 4.4.4]{TesisDaniel}

Let $\b$ be a basic $D$-sequence with bounded ratios. We define on $\Z$ the topology $$\deltab:=\sup\{\tauc:\mathbf{c}\in\Dlinfty(\b)\}.$$

\end{definition}

{\noindent\it Proof of Lemma \ref{LemaDualContable}:}

Since $\lambdab<\Gammab$, it is clear that $\Zbinfty\leq(\Z,\Gammab)^\wedge$.

By definition, it is clear that $\Gammab\leq\deltab$. Hence $(\Z,\Gammab)^\wedge\leq(\Z,\deltab)^\wedge=\Zbinfty$.

\QED \smallskip

Now we set some interesting questions related to the Mackey Problem and $\Gammab$:

\begin{openquestion}

Is $\Gammab=\deltab$?

\end{openquestion}

\begin{openquestion}

Is $\Gammab$ (or $\deltab$, if different) the Mackey topology for $(\Z,\lambdab)$?

\end{openquestion}



The following lemmas are easy to prove, even for a topological space instead of a topological group.

\begin{lemma}

Let $K$ be a countably infinite compact subset of a Hausdorff space. Then $K$ is first countable, and, hence, every accumulation point in $K$ is the limit point of some sequence.

\end{lemma}

\begin{proof}

Let $K = \{x_n, n \in \N\}$.  Fix $a \in K$.  Consider for each $ x_n \in K$  a closed neighborhood of $a$ (this is possible because $K$ is regular) say $V_n \subset X \setminus {x_n}$.  The set $\{V_n, n \in \N\}$ and its finite intersections constitute  a neighborhood basis for $a\in K$.

In fact, if $W$ is any open neighborhood of $a$, we have a family of closed sets  given by $\{X \setminus W, V_n, n \in \N\}$ with empty intersection. Since $K$ is compact, the mentioned family cannot have  the finite intersection property. Thus, there is a finite subfamily $\{ V_j, j \in F\}$ such that $(X\setminus W) \cap_{i \in F} V_i = \emptyset$.

\end{proof}

\begin{lemma}\label{lemma4.7}

If a countable topological  group $G$ has no nontrivial convergent sequences, then it does not have infinite compact subsets either.

\end{lemma}

\begin{proof}

 Suppose, by contradiction, that $K\subset G$ is an infinite compact subset. Since $K$ is first countable and Hausdorff, any accumulation point is the limit of a sequence in $K$. Since$K$ is infinite and compact, it has an accumulation point, contradicting the fact that there does not exist convergent sequences in $G$ (nor in $K$). Hence any compact subset must be finite.

\end{proof}

\begin{proposition}\label{noreflex}

 Let $\b$ be a basic $D$-sequence with bounded ratios, then the group $G=(\Z,\Gammab)$ is not reflexive. In fact, the bidual $G^{\wedge\wedge}$ can be identified with $\Z$ endowed with the discrete topology.

\end{proposition}

\begin{proof}

The dual group of $(\Z, \Gammab)$ has supporting set $\Zbinfty$. The $\Gammab$-compact subsets of $\Z$ are finite by Lemma $\ref{lemma4.7}$, therefore, the dual group carries the pointwise convergence topology. Thus, $G^\wedge$ is exactly $\Zbinfty$ with the topology induced by the euclidean of $\T$. By \cite[4.5]{tesislydia}, the group $\Zbinfty$ has the same dual group as $\T$, namely, $\Z$ with the discrete topology. We conclude that the canonical mapping $\alpha_G$ is an open non-continuous isomorphism.

\end{proof}

\begin{corollary}

Since $(\Z,\Gammab)^{\wedge\wedge}$ is discrete, the homomorphism $\alpha_{(\Z,\Gammab)}$ is not continuous. However, it is an open algebraic isomorphism.

\end{corollary}

\begin{proposition}

Let $\b\in\Dlinfty$ and $G_\gamma:=(\Z,\taub)$. Then $\alpha_{G_\gamma}$ is a non-surjective embedding from $G_{\gamma}$ into $G_\gamma^{\wedge\wedge}$.

\end{proposition}

\begin{proof}

Since $\taub$ is metrizable, $\alpha_{G_\gamma}$ is continuous. The fact that $\taub$ is locally quasi-convex and $\Zbinfty$ separates points of $\Z$ imply that $\alpha_{G_\gamma}$ is injective and open in its image \cite[6.10]{tesislydia}. On the other hand, $\alpha_{G_\gamma}$ is not onto, for otherwise $G_\gamma$ would be reflexive. However, a non-discrete countable metrizable group cannot be reflexive. Indeed, the dual of a metrizable group is a $k$-space (\cite{tesislydia,Cha1998}) and the dual group of a $k$-space is complete. Hence, the original group must be complete as well. By Baire Category Theorem, the only metrizable complete group topology on a countable group is the discrete one.

\end{proof}

\section{The complete topology $\Graev{b_n}$.}\label{SeccionGraev}

In this section we consider some special families of topologies which were introduced by Graev and deeply studied in the group of the integers by Protasov and Zelenyuk (\cite{ProZel1990,ProZel1999}). We need first some definitions about sequences. All sequences considered will be without repeated terms.

\begin{definition}

Let $G$ be a group and $\mathbf{g}=(g_n)\subset G$ a sequence of elements of $G$. We say that $\mathbf{g}$ is a $T$-sequence if there exists a \textbf{Hausdorff} group topology $\tau$ such that $g_n\stackrel{\tau}{\rightarrow}0$.

\end{definition}

\begin{lemma}

Let $\mathbf{g}\subset G$ be a $T$-sequence, then there exists the finest group topology $\mathcal{T}_{\{g_n\}}$ satisfying that $g_n\stackrel{\Graev{g_n}}{\rightarrow}0$.

\end{lemma}

Since we can describe $\Graev{g_n}$ we are interested in finding a suitable neighborhood basis for this topology.

\begin{definition}\cite{ProZel1999}

Let $G$ be a group and let $a=(a_n)$ be a $T$-sequence in $G$ and $(n_i)_{i\in \N}$ a sequence of natural numbers. We define:

\begin{itemize}

\item $A^*_m:=\{\pm a_n|n\geq m\}\cup\{0_G\}$.

\item $A(k,m):=\{g_0+\cdots +g_k| g_i\in A^*_m\, i\in\{0,\dots ,k\}\}$.

\item $[n_1,\dots, n_k]:=\{g_1+\cdots + g_k: g_i\in A^*_{n_i}, i=1,\dots, k\}$.

\item $\dis V_{(n_i)}=\bigcup^\infty_{k=1}[n_1,\dots,n_k]$.

\end{itemize}
\end{definition}

\begin{proposition}

The family $\{V_{(n_i)}:(n_i)\in \N^\N\}$ is a neighborhood basis of $0_G$ for a group topology $\Graev{a_n}$ on $G$, which is the finest among all those satisfying $a_n\rightarrow 0$. The symbol $G_{\{a_n\}}$ will stand for the group $G$ endowed with $\mathcal{T}_{\{a_n\}}$.

\end{proposition}

Next, we include a theorem of great interest from $\cite{ProZel1999}$:

 \begin{theorem}\cite[Theorem 2.3.11]{ProZel1999}\label{TopologiasDeProtasovSonCompletas}

 Let $(g_n)$ be a $T$-sequence on a group $G$. Then, the topology $\Graev{g_n}$ is complete.

 \end{theorem}

\begin{corollary}

Let $\b$ be a $D$-sequence and $c\in\Dlinfty(\b)$. Then $\tauc\neq\Graev{b_n}$.

\end{corollary}

Applying the Baire Category Theorem, we can obtain the following corollary:

\begin{corollary}

Let $G$ be a countable group. For any $T$-sequence $\mathbf{g}$, the topology $\Graev{g_n}$ is not metrizable.

\end{corollary}

Since we are mainly interested on topologies on the group of integers, we include the following results:

\begin{proposition}\cite[Theorem 2.2.1]{ProZel1999}

If $\lim_{n\rightarrow\infty}\frac{a_{n+1}}{a_n}=\infty$ then $(a_n)$ is a $T$-sequence in $\Z$.

\end{proposition}

This proposition provides a condition to prove that the following sequences are $T$-sequences: $(b_n)=(2^{n^2})$, $(\ro_n) = (p_1\cdot\dots\cdot p_n)$.

\begin{proposition}\cite[Theorem 2.2.3]{ProZel1999}

If $lim_{n\rightarrow \infty}\frac{a_{n+1}}{a_n}=r$ and $r$ is a transcendental number then $(a_n)$ is a $T$-sequence in $\Z$.

\end{proposition}

Let $\b$ be a $D$-sequence. It is clear (by Lemma \ref{bConvergeEnLambdab}) that $\b$ is a $T$-sequence. Hence, the topology $\Graev{b_n}$ can be constructed on $\Z$. By Theorem \ref{TopologiasDeProtasovSonCompletas}, it is a complete topology.

The following properties of $\Z_{\{a_n\}}$  are interesting by themselves. The first one solves a problem by Malykhin and the second one states that the dual group is metrizable and a $k$-space.

\begin{remark}\label{kW}

Let $(a_n)$ be a $T$-sequence on $\Z$. Then:

\begin{itemize}

\item[(1)]   $\Z_{\{a_n\}}$ is sequential, but not Frechet-Urysohn (\cite[Theorem 7 and Theorem 6]{ProZel1990}).

\item[(2)] $\Z_{\{a_n\}}$ is a $k_{\omega}$-group (\cite[Corollary 4.1.5]{ProZel1999} ).
\end{itemize}
\end{remark}

We prove now that, for a $D$-sequence $\b$ with bounded ratios, the locally quasi-convex modification of $\Graev{b_n}$ is $\lambdab$. This fact is also proved in \cite[Proposition 2.5]{DikGabTar2014}, but we give here a direct proof. First we study some numerical results in the group $\T$:

\begin{lemma}\label{lema1}

Let $k\in\T$ and $m\in\N$ satisfy $k,2k,\dots,mk\in\T_+$. If $mk\in\T_n$, then $k\in\T_{nm}$.

\end{lemma}

\begin{proof}

Write $k=z+y$ where $z\in\Z$ and $y\in\left(-\frac{1}{2},\frac{1}{2}\right]$. Since $k\in\T_+$, we have $\abs{y}\leq\frac{1}{4}$.

Suppose that $k\notin\T_{nm}$. This means that $\frac{1}{4nm}<\abs{y}\leq\frac{1}{4}$. Now, $mk=mz+my$. Since $\frac{1}{4n}<\abs{my}\leq\frac{m}{4}$ and $mk\in\T_n$, we have that $\abs{my}\geq 1-\frac{1}{4n}$. Hence, there exists $N\leq k$ such that $\frac{1}{4}<\abs{Ny}<\frac{3}{4}$. Consequently, $Nk\notin\T_+$.

\end{proof}

\begin{lemma}\label{lema2}

Let $k\in\T$ and let $m_1,m_2,\dots,m_n\in\N\setminus\{1\}$. If $$k,2k,3k,\dots,m_1k,2m_1k,\dots,m_1m_2k,\dots,\left(\prod_{i=1}^{n-1}m_i\right)k,2\left(\prod_{i=1}^{n-1}m_i\right)k,\dots,\left(\prod_{i=1}^{n}m_i\right)k\in\T_+,$$ then $$k\in\T_{\prod_{i=1}^{n}m_i}.$$

\end{lemma}

\begin{proof}

Since $\left(\prod_{i=1}^{n-1}m_i\right)k,2\left(\prod_{i=1}^{n-1}m_i\right)k,\dots\left(\prod_{i=1}^{n}m_i\right)k\in\T_+$ and $\left(\prod_{i=1}^{n}m_i\right)k\in\T_1$, by Lemma \ref{lema1}, we have that $\left(\prod_{i=1}^{n-1}m_i\right)k\in\T_{m_n}$.

Since $\left(\prod_{i=1}^{n-2}m_i\right)k,2\left(\prod_{i=1}^{n-2}m_i\right)k,\dots\left(\prod_{i=1}^{n-1}m_i\right)k\in\T_+$ and $\left(\prod_{i=1}^{n-1}m_i\right)k\in\T_{m_n}$, by Lemma \ref{lema1}, we have that $\left(\prod_{i=1}^{n-2}m_i\right)k\in\T_{m_nm_{n-1}}$.

Repeating the argument $n$ times, we obtain that $k\in\T_{\prod_{i=1}^{n}m_i}$.

\end{proof}

\begin{lemma}\label{lema3}

Let $k\in\T$ and let $(m_i)\in\N^\N$ a sequence of natural numbers satisfying $m_i>1$ for all $i\in\N$. If $$k,2k,3k,\dots,m_1k,2m_1k,\dots,m_1m_2k,\dots,\left(\prod_{i=1}^{n-1}m_i\right)k,2\left(\prod_{i=1}^{n-1}m_i\right)k,\dots,\left(\prod_{i=1}^{n}m_i\right)k,\dots\in\T_+$$ then $k=0+\Z$.

\end{lemma}

\begin{proof}

By Lemma \ref{lema2}, we have that $k\in\bigcap_{N>0}\T_{\prod_{i=1}^{N}m_i}$. Hence, $k=0+\Z$.

\end{proof}

Now we can study the polar of the neighborhoods in the topology $\Graev{b_n}$.

\begin{proposition}

Let $\b$ be a $D$-sequence with bounded ratios and $V_{(n_i)}$ a basic neighborhood of $\Graev{b_n}$. Then, there exists $N\in\N$ such that $$A:=\{b_N,2b_N,\dots,b_{N+1},2b_{N+1},\dots,b_{N+2},2b_{N+2},\dots\}\in V_{(n_i)}.$$
As a consequence $V_{(n_i)}\tr$ are finite subsets of $\T$, for any sequence $(n_i)$.

\end{proposition}

\begin{proof}

Since $\b$ has bounded ratios, there exists $L\in\N$ such that $q_n\leq L$ for all $n\in\N$.

Let $n_1,n_2,\dots, n_L$ be the indexes satisfying that $n_i\leq n_j$ if $i\in\{1,2,\dots,L\}$ and $j\notin\{1,2,\dots,L\}$. Let $N:=\max\{n_1,n_2,\dots,n_L\}$. Then $b_N,b_{N+1},\dots\in A^*_{n_i}$ for $i\in\{1,2,\dots,L\}$ and $$A\subset \{b_N,2b_N,\dots, Lb_N,b_{N+1},2b_{N+1},\dots,Lb_{N+1},\dots\}\subset[n_1,n_2,\dots,n_L]\subset V_{(n_i)}.$$

Let $\chi\in V_{(n_i)}\tr$. Since $A\subset V_{(n_i)}$, it is clear that $V_{(n_i)}\tr\subset A\tr$. Define $k:=\chi(b_N)$. Then $\chi(A)=\{k,2k,\dots q_Nk, 2q_Nk,\dots,q_Nq_{N+1}k,\dots\}$. Since $\chi\in A\tr$, we have $\chi(A)\subset\T_+$. By Lemma \ref{lema3}, the fact $k=\chi(b_N)=0+\Z$ follows. Hence $\chi\in \Z(b_N)$ and $V_{(n_i)}\tr$ is finite.

\end{proof}

Finally, we can prove that for a $D$-sequence, $\b$, with bounded ratios, the locally quasi-convex modification of $\Graev{b_n}$ is $\lambdab$.

\begin{corollary}\label{DualesCoincidenAlgebraicamente}

Let $\b$ be a $D$-sequence with bounded ratios. Then $\Z_{\{b_n\}}^\wedge=\Zbinfty$.

\end{corollary}

\begin{proof}

Since $\Z_{\{b_n\}}=\bigcup_{(n_i)}V_{(n_i)}\tr$, we have that $\Z_{\{b_n\}}^\wedge\leq\Zbinfty$.

Since $\b\rightarrow0$ in $\lambdab$, $\lambdab\leq\Graev{b_n}$ and $\Z_{\{b_n\}}^\wedge\geq\Zbinfty$.

\end{proof}

\begin{theorem}

Let $\b$ a $D$-sequence with bounded ratios. Then the locally quasi-convex modification of $\Graev{b_n}$ is $\lambdab$.

\end{theorem}

\begin{proof}

Since the polar sets of the basic neighborhoods of $\Graev{b_n}$ are finite, the equicontinuous subsets in the dual group are finite as well. Hence the locally quasi-convex modification of $\Graev{b_n}$ is precompact. Since $\Graev{b_n}$ and its locally quasi-convex modification are compatible topologies, it follows that $\left(\Graev{b_n}\right)_{lqc}=\lambdab$.

\end{proof}

\begin{corollary}\label{GraevNoLQC}

Let $\b$ a $D$-sequence with bounded ratios. Then $\Graev{b_n}$ is not locally quasi-convex.

\end{corollary}

\begin{proof}

If $\Graev{b_n}$ were locally quasi-convex, then $\left(\Graev{b_n}\right)_{lqc}=\Graev{b_n}$, but $\left(\Graev{b_n}\right)_{lqc}$ is not complete and $\Graev{b_n}$ is complete.

\end{proof}

\begin{remark}

Although the examples in Corollary \ref{GraevNoLQC} are not locally quasi-convex, there exists examples of Graev-type topologies on $\Z$ which are locally quasi-convex. Indeed, in \cite{Gab2010} it is proved that there exist Graev-type topologies on $\Z$ which are reflexive (hence locally quasi-convex).

\end{remark}

\begin{remark}

Corollary \ref{GraevNoLQC} gives a family of examples of:

\begin{itemize}
\item Complete topologies whose locally quasi-convex modifications are not complete.

\item Non-metrizable topologies whose locally quasi-convex modifications are metrizable.
\end{itemize}

\end{remark}

\begin{proposition}\label{DualesCoincidenTopologicamente}

The dual group of $\Z_{\{b_n\}}$ coincides algebraically and topologically with the dual group of $(\Z,\lambdab)$; that is, the dual group of $\Z_{\{b_n\}}$ is $\Zbinfty$ endowed with the discrete topology.

\end{proposition}
\begin{proof}

 From Corollary \ref{DualesCoincidenAlgebraicamente} the dual group of $\Z_{\{b_n\}}$ is $\Z_{\{b^n\}}^\wedge=\Zbinfty$. Item $(2)$ in Remark \ref{kW} implies that $\Z_{\{b_n\}}^\wedge$ is metrizable and complete. By Baire Category Theorem and since $\Z_{\{b_n\}}^\wedge$ is countable, we get that it is discrete. Hence, it coincides topologically with $(\Z,\lambdab)^\wedge$.

\end{proof}

Now we answer the question whether $\Z_{\{b_n\}}$ is $\MAP$-Mackey.

\begin{proposition}

The topology $\Z_{\{b_n\}}$ is not $\MAP$-Mackey.

\end{proposition}

\begin{proof}

If $\Graev{b_n}$ were Mackey, the condition $\Gammab\leq\Graev{b_n}$ should hold. But, the topology $\Gammab$ has no convergent sequences and $b_n\rightarrow 0$ in $\Graev{b_n}$.

\end{proof}

\begin{openquestion}

Let $(g_n)$ a $T$-sequence in $\Z$. If $\Graev{g_n}$ is locally quasi-convex, is $\Graev{g_n}$ the Mackey Topology?

\end{openquestion}

\end{document}